\documentclass[english]{amsart}
\usepackage{amsfonts}
\usepackage{amsthm}
\usepackage{amsmath}
\usepackage{amssymb}
\usepackage{mathrsfs}
\usepackage{amscd}
\usepackage{mathdots}
\usepackage[all]{xy}
\usepackage[bf, small]{titlesec}

\titleformat{\section}{\bf\large}{\thesection.\,}{0.24em}{}
\titlespacing{\section}{0cm}{*2}{*1.1}
\titleformat{\subsection}[runin]{\bf}{\thesubsection.\,}{0.5em}{}
\titlespacing{\subsection}{0cm}{*2}{*1.5}
\titleformat{\subsubsection}[runin]{\rm}{(\thesubsubsection)}{0.5em}{}
\titlespacing{\subsubsection}{0cm}{*1.1}{*1.5}

\newcommand{\V}{\mathcal {V}}
\newcommand{\ES}{\mathscr{S}}

\newcommand{\dr}{\mathrm{dR}}

\newcommand{\INT}{\mathbb{Z}}
\newcommand{\RAT}{\mathbb{Q}}
\newcommand{\INTP}{\mathbb{Z}_{(p)}}

\newcommand{\Isom}{\mathbf{Isom}}

\newcommand{\Sh}{\mathrm{Sh}}
\newcommand{\Hdr}{\mathrm{H}^1_{\mathrm{dR}}}

\normalsize \relpenalty=10000 \binoppenalty=10000

\begin{document}
\newtheorem{theorem}[subsection]{Theorem}
\newtheorem{lemma}[subsection]{Lemma}
\newtheorem{proposition}[subsection]{Proposition}
\newtheorem{corollary}[subsection]{Corollary}
\theoremstyle{plain}
\newtheorem{introth}{Theorem}
\renewcommand{\theintroth}{\Alph{introth}}
\theoremstyle{definition}
\newtheorem{definition}[subsection]{Definition}
\theoremstyle{definition}
\newtheorem{construction}[subsection]{Construction}
\theoremstyle{definition}
\newtheorem{notations}[subsection]{Notations}
\theoremstyle{definition}
\newtheorem{asp}[subsection]{Assumption}
\theoremstyle{definition}
\newtheorem{set}[subsection]{Setting}
\theoremstyle{remark}
\newtheorem{remark}[subsection]{Remark}
\theoremstyle{remark}
\newtheorem{example}[subsection]{Example}
%\titleformat{\subsubsection}{\raggedright\large\bfseries}{(\thesubsubsection)}{0pt}{}

%%%make equations numbered as subsections
\makeatletter
\newenvironment{subeqn}{\refstepcounter{subsubsection}
$$}{\leqno{\rm(\thesubsubsection)}$$\global\@ignoretrue}
\makeatother

\author{Chao Zhang}

\address{Institute of Mathematics, Academia Sinica, 6F, Astronomy-Mathematics Building, No. 1, Sec. 4, Roosevelt Road, Taipei 10617, Taiwan.}

\email{zhangchao1217@gmail.com}

\title[Compatibility of integral models]{Compatibility of certain integral models of Shimura varieties of abelian type}

\maketitle  \setcounter{tocdepth}{3}
\begin{abstract}
For a prime $p>2$, Kisin and Pappas constructed parahoric integral models at $p$ for Shimura varieties attached to Shimura data $(G,X)$ of abelian type such that $G$ splits over a tamely ramified extension of $\RAT_p$. A certain auxiliary data has to be chosen in their constructions. In this note, we will show that the parahoric integral models are actually independent of the choices of the auxiliary data. We also get partial results on extending morphisms of Shimura varieties to those of parahoric integral models.
\end{abstract}

\section[Introduction]{Introduction}
Let $(G,X)$ be a Shimura datum, and $\Sh_K(G,X)_{\mathbb{C}}$ be the Shimura variety attached to a compact open subgroup $K\subseteq G(\mathbb{A}_f)$. By works of Shimura, Deligne, Borovoi and Milne, $\Sh_K(G,X)_{\mathbb{C}}$ has a unique canonical model $\Sh_K(G,X)_{\mathrm{E}}$ over the reflex field $\mathrm{E}$. One also needs to study integral models of $\Sh_K(G,X)_{\mathrm{E}}$ for applications to arithmetic problems, e.g. the Langlands's program.

We fix a prime number $p>2$ as well as a place $v$ of $\mathrm{E}$ over $p$ once and for all. The ring $O_{\mathrm{E},v}$, the $v$-adic completion of $O_\mathrm{E}$, will be denoted by $O_E$; and $O_E[\frac{1}{p}]$ will be denoted by $E$. We will write $\Sh_K(G,X)$ for $\Sh_K(G,X)_{\mathrm{E}}\otimes_{\mathrm{E}}E$. If $(G,X)$ is of abelian type and $K$ is of the form $K_pK^p$ with $K^p$ small enough, there are many progresses in constructing integral models $\ES_K(G,X)$ for $\Sh_K(G,X)$ in recent years.

\subsection{} If $K_p$ is hyperspecial, then $E/\RAT_p$ is unramified, and $\ES_K(G,X)$ is constructed by Vasiu and Kisin separately. We refer to \cite{CIMK} for details of the constructions, but just mention the following properties.
%\begin{enumerate}
%\item For $K^p$ smooth enough, $\ES_K(G,X)$ is smooth over $O_E$. The $\ES_K(G,X)$s form a inverse system with finite \'{e}tale transition morphisms.
%\item
%\item
%\end{enumerate}
\subsubsection{}\label{hypsp-sm} For $K^p$ small enough, $\ES_K(G,X)$ is smooth over $O_E$. As $K^p$ varies, the $\ES_K(G,X)$s form a inverse system  with finite \'{e}tale transition morphisms.
\subsubsection{}\label{hypsp-ext} The inverse limit $\ES_{K_p}(G,X):=\varprojlim_{K^p}\ES_K(G,X)$ is such that for any formally smooth $O_E$-algebra $R$ which is also regular, the natural map $$\ES_{K_p}(G,X)(R)\rightarrow \ES_{K_p}(G,X)(R[1/p])$$ is bijective.

These two properties characterize $\ES_K(G,X)$ uniquely. Moreover, they also imply that the formulism is functorial in the following sense.
\subsubsection{}\label{hypsp-funct} Let $(G',X')$ be a Shimura datum of abelian type with $K'_p\subseteq G'(\RAT_p)$ hyperspecial, and  $f:(G,X)\rightarrow (G',X')$ be a morphism of Shimura data such that $f(K)\subseteq K'=K'_pK'^p$. Then the morphism $f:\Sh_{K}(G,X)\rightarrow \Sh_{K'}(G',X')_{E}$ extends to a morphism $\ES_{K}(G,X)\rightarrow \ES_{K'}(G',X')_{O_E}$.

\subsection{} If $K_p$ is, more generally, parahoric, $\Sh_{K}(G,X)$ is expected to have an integral model $\ES_K(G,X)$ whose singularity is controlled by a certain local model $\mathrm{M}_{G,X}^{\mathrm{loc}}$. The constructions of $\ES_K(G,X)$ and $\mathrm{M}_{G,X}^{\mathrm{loc}}$ are still vastly open. In this paper, we will discuss some known cases established in \cite{Paroh}. To explain the problems considered in this paper as well as the results, we need to recall the construction of $\ES_K(G,X)$ very briefly.

\subsubsection{}\label{intro--recall Hodge} If $(G,X)$ is of Hodge type, such that $G_{\RAT_p}$ splits over a tamely ramified extension of $\RAT_p$ and that $p\nmid |\pi_1(G_{\mathbb{Q}_p}^{\mathrm{der}})|$, one can \emph{choose} a good symplectic embedding in the sense of \S\ref{para int setting}, and take $\ES_K(G,X)$ to be the normalization of the closure of $\Sh_K(G,X)$ in the corresponding moduli space of abelian schemes. One of the main results in \cite{Paroh} (recalled in Theorem \ref{result P-Z and K-P} here) claims that it dose have the expected singularity.

\subsubsection{}\label{intro--recall Ab} If $(G,X)$ is of abelian type, such that $G_{\RAT_p}$ splits over a tamely ramified extension of $\RAT_p$, one can \emph{choose} a Shimura datum $(G_1,X_1)$ of Hodge type as in Lemma \ref{lemma--choose of Hodge cover}, and construct $\ES_K(G,X)$ using the integral model attached to $(G_1,X_1)$. We refer to \S\ref{setting--ab type} for more details, but only mention that $(G_1,X_1)$ satisfies the hypothesis in (\ref{intro--recall Hodge}), and the integral model is constructed as in loc. cit.

\subsection{} Now we can describe our problems. Our main concern is whether these integral models are independent of the choices made when constructing them. More precisely, it consists of the following three questions.

\subsubsection{}\label{prob--Hodge type} If $(G,X)$ is of Hodge type as in (\ref{intro--recall Hodge}), whether $\ES_K(G,X)$ is independent of choices of symplectic emdeddings there.
\subsubsection{}\label{prob--Ab type} If $(G,X)$ is of abelian type as in (\ref{intro--recall Ab}), whether $\ES_K(G,X)$ is independent of choices of the $(G_1,X_1)$s there.
\subsubsection{}\label{prob--Hodge to Hodge} If $(G,X)$ is of Hodge type as in (\ref{intro--recall Hodge}), one can follow either (\ref{intro--recall Hodge}) or (\ref{intro--recall Ab}) to construct an integral model for $\Sh_K(G,X)$. The question is then whether these two constructions coincide.

The main result of this paper is as follows. We refer to Theorem \ref{uniq int model}, Theorem \ref{Thm--indep ab type} and Theorem \ref{Thm--indep ab type 2} for the proof.
\begin{introth}The integral model $\ES_K(G,X)$ is independent of the choices made above when constructing it. More precisely, the problems (\ref{prob--Hodge type}), (\ref{prob--Ab type}) and (\ref{prob--Hodge to Hodge}) all have positive answers.
\end{introth}

%As a by product, we also obtain partial results when morphisms of Shimura varieties extends to integral models.
Let $(G',X')$ be another Shimura data of abelian type, such that $G'_{\RAT_p}$ splits over a tamely ramified extension of $\RAT_p$. Let $K'\subseteq G'(\mathbb{A}_f)$ be of the form $K'_pK'^p$ with $K'_p\subseteq G'(\RAT_p)$ parahoric, and $f:(G,X)\rightarrow (G',X')$ be a morphism of Shimura data with $f(K)\subseteq K'$. One then asks whether (or when) the morphism $\Sh_{K}(G,X)\rightarrow \Sh_{K'}(G',X')_{E}$ extends to a morphism $\ES_{K}(G,X)\rightarrow \ES_{K'}(G',X')_{O_E}$.
It could be expected that it extends if $G\rightarrow G'$ extends to a homomorphism of the corresponding parahoric group schemes. Unfortunately, we can only prove some special and weaker results which could be viewed as supporting evidences for this expectation. We collect some of our results here and refer to Proposition \ref{extn morp on int}, Proposition \ref{prop--funct ab type} and \S\ref{remark--when funct works} for more details.

\begin{introth}Notations as above, the morphism $f:\Sh_{K}(G,X)\rightarrow \Sh_{K'}(G',X')_{E}$ extends to a morphism $\ES_{K}(G,X)\rightarrow \ES_{K'}(G',X')_{O_E}$ in the following two cases.
\begin{enumerate}
\item $(G,X)$ and $(G',X')$ are of Hodge type, and $f:G\rightarrow G'$ extends to a homomorphism of the corresponding Bruhat-Tits group schemes.
\item $f:G\rightarrow G'$ is surjective with $(G^\mathrm{ad},X^\mathrm{ad})$ having no factors of type $\mathrm{D}^\mathbb{H}$, and $K_p\rightarrow K'_p$ is induced by $f$ in the sense of (\ref{case--induced by surj}).
\end{enumerate}
\end{introth}

We learned during the 2019 Oberwolfach workshop ``Arithmetic of Shimura Varieties'' that G. Pappas also proves (\ref{prob--Hodge type}) positively. It seems that these two proofs follow the same strategy, but they are still different from each other in many ways.

%Let $G$ be a reductive group over $\mathbb{Q}$, we will write $\mathcal {B}(G,\mathbb{Q}_p)$ for the Bruhat-Tits Building of $G_{\mathbb{Q}_p}$. For $\mathfrak{b}\in \mathcal {B}(G,\mathbb{Q}_p)$, we write $G_{\INT_p}$ for the attached Bruhat-Tits group scheme, and $G_{\INT_p}^\circ$ for its connected component of the identity. They are both linear algebraic groups over $\mathbb{Z}_p$ with generic fibers isomorphic to $G_{\mathbb{Q}_p}$. The group scheme $G_{\INT_p}$ (resp. $G_{\INT_p}^\circ$) descends to a (unique) smooth $\INTP$-model of $G$, denoted by $G_{\INTP}$ (resp. $G_{\INTP}^\circ$).

%Let $(G,X)$ be a Shimura datum with reflex field $\mathrm{E}$.  Fixing $K_p:=G_{\INT_p}(\mathbb{Z}_p)$ and $K_p^\circ:=G_{\INT_p}^\circ(\mathbb{Z}_p)$, for $K^p\subseteq G(\mathbb{A}_f^p)$ small enough, we set $K:=K_pK^p$  and $K^\circ:=K_p^\circ K^p$. We are interested in a certain integral model of $\Sh_{K^\circ}(G,X)_E$ when it is of abelian type, and also of $\Sh_K(G,X)_E$ when it is of Hodge type.

\section[Integral models of Shi]{Integral models of Shimura varieties}\label{sec--int mod}

Let $G$ be a reductive group over $\mathbb{Q}$, we will write $\mathcal {B}(G,\mathbb{Q}_p)$ for the (extended) Bruhat-Tits Building of $G_{\mathbb{Q}_p}$. For $\mathfrak{b}\in \mathcal {B}(G,\mathbb{Q}_p)$, we write $G_{\INT_p}$ for the attached Bruhat-Tits group scheme, and $G_{\INT_p}^\circ$ for its connected component of the identity (i.e. the parahoric group scheme). They are both linear algebraic groups over $\mathbb{Z}_p$ with generic fibers isomorphic to $G_{\mathbb{Q}_p}$. The group scheme $G_{\INT_p}$ (resp. $G_{\INT_p}^\circ$) descends to a (unique) smooth $\INTP$-model of $G$, denoted by $G_{\INTP}$ (resp. $G_{\INTP}^\circ$). This $\INTP$-model will also be called the Bruhat-Tits group scheme (resp. parahoric group scheme) attached to $\mathfrak{b}$. Let $f:G\rightarrow G'$ be a homomorphism of reductive groups with $G'_{\INT_p}$, $G'_{\INTP}$ obtained similarly, then any homomorphism (if exists) $G_{\INT_p}\rightarrow G'_{\INT_p}$ extending $f_{\RAT_p}$ descends to $\INTP$. The same statement holds if we pass to the identity components.

Let $(G,X)$ be a Shimura datum with reflex field $\mathrm{E}$. We will fix a place $v$ over $p$ once and for all. The ring $O_{\mathrm{E},v}$, the $v$-adic completion of $O_\mathrm{E}$, will be denoted by $O_E$; and $O_E[\frac{1}{p}]$ will be denoted by $E$. Fixing $K_p:=G_{\INT_p}(\mathbb{Z}_p)$ and $K_p^\circ:=G_{\INT_p}^\circ(\mathbb{Z}_p)$, for $K^p\subseteq G(\mathbb{A}_f^p)$ small enough, we set $K:=K_pK^p$  and $K^\circ:=K_p^\circ K^p$. We are interested in a certain integral model of $\Sh_{K^\circ}(G,X)_E$ when it is of abelian type, and also of $\Sh_K(G,X)_E$ when it is of Hodge type.

\subsection{}\label{para int setting} Let $(G,X)$ be a Shimura datum of Hodge type. We assume, in addition, that
\begin{subeqn}\label{conditon--Hodge type}
G_{\mathbb{Q}_p} \text{ splits over a tamely ramified extension and } p\nmid |\pi_1(G_{\mathbb{Q}_p}^{\mathrm{der}})|.
\end{subeqn}

By \cite[\S4.1.6]{Paroh}, there is a symplectic embedding $i:(G,X)\rightarrow \mathrm{GSp}(V,\psi)$ with a  $\mathbb{Z}$-lattice $V_{\INT}\subseteq V$, such that $V_{\INT}\subseteq V_{\INT}^\vee$; and that the Zariski closure of $G$ in $\mathrm{GL}(V_{\mathbb{Z}_{(p)}})$, is $G_{\INTP}$. Here for a commutative ring $R$, we write $V_R$ for $V_{\INT}\otimes R$. Let $d=| V_{\INT}^\vee/V_{\INT}|$, $g=\frac{1}{2}\mathrm{dim}V$, $L=L_pL^p$ where $L_p$ is the subgroup of $\mathrm{GSp}(V,\psi)(\mathbb{Q}_p)$ leaving $V_{\mathbb{Z}_p}$ stable, and $L^p$ is a compact open subgroup of $\mathrm{GSp}(V,\psi)(\mathbb{A}_f^p)$ containing $K^p$, leaving $V_{\widehat{\mathbb{Z}}^p}$ stable and small enough. The symplectic embedding induces a morphism $i:\Sh_K(G,X)_E\rightarrow \mathscr{A}_{g,d,L/O_E}$, whose scheme theoretic image is denoted by $\ES^-_K(G,X)$. Let $\ES_K(G,X)$ be the normalization of $\ES^-_K(G,X)$, and $\ES_{K^\circ}(G,X)$ be the normalization of $\ES_{K}(G,X)$ with respect to the covering $\Sh_{K^\circ}(G,X)\rightarrow \Sh_{K}(G,X)$. They are integral models of $\Sh_K(G,X)$ and $\Sh_{K^\circ}(G,X)$ respectively, with a finite surjective morphism $\ES_{K^\circ}(G,X)\rightarrow \ES_{K}(G,X)$.

\subsection{} We recall the constructions of local models. Instead of using the original constructions in \cite{local model P-Z}, we will follow \cite[\S4.1.5]{Paroh}. Notations and assumptions as before, the Shimura datum $(G,X)$ gives a $G(\overline{\mathbb{Q}}_p)$-conjugacy class $\{\mu\}$ of cocharacters. Its induced $\mathrm{GL}(V_{\mathbb{Z}_p})(\overline{\mathbb{Q}}_p)$-conjugacy class contains a cocharacter $\mu'$ defined over $\mathbb{Z}_p$. Let $P$ (resp. $P'$) be the parabolic subgroup of $G_{\overline{\mathbb{Q}}_p}$ (resp. $\mathrm{GL}(V_{\mathbb{Z}_p})$) with non-negative weights with respect to $\mu'$ (resp. $\mu$), we denote by $\mathrm{M}_{G,X}^{\mathrm{loc}}$ the closure of the $G$-orbit of $y$ in $\mathrm{GL}(V_{\mathbb{Z}_p})/P'$, where $y$ is the point corresponding to $P$. Then $\mathrm{M}_{G,X}^{\mathrm{loc}}$ is defined over $O_E$ and equipped with an action of $G_{\INT_p}$. It is called the \emph{local model}.

To explain properties of $\ES_K(G,X)$, $\ES_{K^\circ}(G,X)$ and $\mathrm{M}_{G,X}^{\mathrm{loc}}$ as well as relations between them, we need the following data. By \cite[Lemma 1.3.2]{CIMK}, there is a tensor $s\in V_{\mathbb{Z}_{(p)}}^\otimes$ defining $G_{\INTP}$. Let $\V$ be $\Hdr(\mathcal {A}/\ES_K(G,X))$, $\V_E$ be $\Hdr(\mathcal {A}/\Sh_K(G,X)_E)$, $\V^1\subseteq \V$ be the Hodge filtration, and $s_{\mathrm{dR},E}\in \V_E^\otimes$ be the section induced by $s$.
\begin{theorem}\label{result P-Z and K-P}
Notations an assumptions as above, we have the followings.

\begin{enumerate}
\item $\mathrm{M}_{G,X}^{\mathrm{loc}}$ is normal with reduced geometric special fiber. Moreover, $\mathrm{M}_{G,X}^{\mathrm{loc}}$ depends only on the pair $(G_{\mathbb{Q}_p}, \{\mu\})$.
\item $s_{\mathrm{dR},E}$ extends to a tensor $s_{\mathrm{dR}}\in \V^\otimes$. The $\ES_K(G,X)$-scheme $\widetilde{\ES}_K(G,X)$ classifying isomorphisms $f:V_{\mathbb{Z}_p}^\vee\rightarrow \V$ mapping $s$ to $s_{\mathrm{dR}}$ is a $G_{\INT_p}$-torsor.
\item The morphism $q:\widetilde{\ES}_K(G,X)\rightarrow \mathrm{GL}(L_{\mathbb{Z}_p})/P'$, $f\mapsto f^{-1}(\V^1)$ factors through $\mathrm{M}_{G,X}^{\mathrm{loc}}$. Moreover, it is smooth.

\item The morphism $\ES_{K^\circ}(G,X)\rightarrow \ES_K(G,X)$ is a finite \'{e}tale covering.
\end{enumerate}
\end{theorem}
\begin{proof}
Statement (1) follows from the construction in \cite{local model P-Z}, statement (2) is \cite[Proposition 4.2.6]{Paroh}, statement (3) is \cite[Theorem 4.2.7]{Paroh}, and statement (4) is \cite[Proposition 4.3.7]{Paroh}.
\end{proof}

\subsection{}\label{setting--ab type} Let $(G_4,X_4)$\footnote{The notation $(G_4,X_4)$ might seem strange, but the readers would see later that this would make our notations coherent.} be a Shimura datum of abelian type. There is a $(G,X)$ of Hodge type together with a central isogeny $G^{\mathrm{der}}\rightarrow G_4^{\mathrm{der}}$ inducing an isomorphism of Shimura data $(G^{\mathrm{ad}},X^{\mathrm{ad}})\rightarrow (G_4^{\mathrm{ad}},X_4^{\mathrm{ad}})$. By \cite[\S2.3]{varideshi}, we can choose $(G,X)$ such that $G^{\mathrm{der}}$ is as follows: $G^{\mathrm{ad}}=\prod_iH_i$ where each $H_i$ is $\mathbb{Q}$-simple and adjoint. Then $H_i=\mathrm{Res}_{F_i/\mathbb{Q}}H'_i$ where $F_i$ is totally real, $H'_i$ is absolutely simple over $F_i$ and adjoint. Let $\widetilde{H}'_i$ be the simply connected cover of $H'_i$, we have $G^{\mathrm{der}}=\prod_i\mathrm{Res}_{F_i/\mathbb{Q}}H'^\sharp_i$, with $H'^\sharp_i:=\widetilde{H}'_i/C_i$ if $(H_i,X_i)$ is of type $\mathrm{D}^{\mathbb{H}}$, here $C_i$ is the kernel of the first fundamental weight restricted to the center of $\widetilde{H}'_i$; and $H'^\sharp_i:=\widetilde{H}'_i$ otherwise. In particular, $\pi_1(G_{\mathbb{Q}_p}^{\mathrm{der}})$ is a 2-group, and trivial if $(G^{\mathrm{ad}},X^{\mathrm{ad}})$ has no factors of type $\mathrm{D}^{\mathbb{H}}$.

We assume from now on that
\begin{subeqn}\label{assump--ab type}
G_{4,\RAT_p}\text{ splits over a tamely ramified extension of }\RAT_p.
\end{subeqn}
Then the $F_i$s above are tamely ramified over $p$.
\begin{lemma} {\rm(\cite[Lemma 4.6.22]{Paroh})}\label{lemma--choose of Hodge cover}
We can choose the $(G,X)$ above satisfying the following additional conditions.
\begin{enumerate}
\item The center $Z_G$ is a torus.
\item $G_{\RAT_p}$ splits over a tamely ramified extension of $\RAT_p$.
\item Any prime $v_4|p$ of $\mathrm{E}_4$ splits completely in $\mathrm{E}\cdot\mathrm{E}_4$, where $\mathrm{E}$ (resp. $\mathrm{E}_4$) is the reflex field of $(G,X)$ (resp. $(G_4,X_4)$).
\end{enumerate}
\end{lemma}

We will fix a choice of $(G,X)$ satisfying the conditions in the above lemma. In particular, condition (\ref{conditon--Hodge type}) is satisfied. For $\mathfrak{b}\in \mathcal {B}(G,\mathbb{Q}_p)$, we denote by $G_{\INTP}$ the attached parahoric group scheme, and $G_{\INTP}^\circ$ (resp. $G_{\INTP}^{\mathrm{der}\circ}$, $G_{\INTP}^{\mathrm{ad}\circ}$) the attached parahoric model of $G$ (resp. $G^{\mathrm{der}}$, $G^{\mathrm{ad}}$). Let $Z_{\INTP}$ be the closure of $Z_G$ in $G_{\INTP}$. It is smooth by \cite[Remark 1.1.8]{Paroh}. Let $G_{\INTP}^{\mathrm{ad}}:=G_{\INTP}/Z_{\INTP}$, by \cite[Lemma 4.6.2 (2)]{Paroh} $G_{\INTP}^{\mathrm{ad}\circ}$ is the identity component of $G_{\INTP}^{\mathrm{ad}}$.
 
We set, as in \cite[\S4.5.6]{Paroh}, that
$$\mathscr{A}(G):=G(\mathbb{A}_f)/Z_G(\RAT)^-*_{G(\RAT)_+/Z_G(\RAT)}G^{\mathrm{ad}}(\RAT)^+$$
$$\mathscr{A}(G)^\circ:=G(\RAT)_+^-/Z_G(\RAT)^-*_{G(\RAT)_+/Z_G(\RAT)}G^{\mathrm{ad}}(\RAT)^+;$$
and as in \cite[\S4.6.3]{Paroh}, that
$$\mathscr{A}(G_{\INTP}):=G(\mathbb{A}_f)/Z_G(\INTP)^-*_{G^\circ(\INTP)_+/Z_G(\INTP)}G^{\mathrm{ad}\circ}(\INTP)^+$$
$$\mathscr{A}(G_{\INTP})^\circ:=G^\circ(\INTP)_+^-/Z_G(\INTP)^-*_{G^\circ(\INTP)_+/Z_G(\INTP)}G^{\mathrm{ad}\circ}(\INTP)^+.$$
We refer to loc. cit. for explanations of these notations.

By \cite[2.1.15]{varideshi} (resp. \cite[Lemma 4.6.4 (2)]{Paroh}), $\mathscr{A}(G)^\circ$ (resp. $\mathscr{A}(G_{\INTP})^\circ$) is the completion of $G^{\mathrm{ad}}(\RAT)^+$ (resp. $G^{\mathrm{ad}\circ}_{\INTP}(\INTP)^+$) with respect to the topology whose open subsets are images of congruence subgroups of $G^{\mathrm{der}}(\RAT)$ (resp. images of intersections of $G^{\mathrm{der}\circ}_{\INTP}(\INTP)_+$ with $p$-congruence subgroups). In particular, it depends only on $G^{\mathrm{der}}$ (resp. $G^{\mathrm{der}\circ}_{\INTP}$).

Let $\mathfrak{b_4}\in \mathcal {B}(G_4,\mathbb{Q}_p)$ be such that its image in $\mathcal {B}(G^{\mathrm{ad}},\mathbb{Q}_p)$ coincides with that of $\mathfrak{b}$. We define $G_{4,\INTP}^\circ$, $\mathscr{A}(G_4)$ and $\mathscr{A}(G_4)^\circ$ in the same way as above, but with $$\mathscr{A}(G_{4,\INTP}):=G_4(\mathbb{A}_f)/Z_{G_4}(\INTP)^-*_{G^\circ_4(\INTP)_+/Z_{G_4}(\INTP)}G^{\mathrm{ad}\circ}(\INTP)^+$$ as in \cite[\S4.6.8]{Paroh}. We will then fix a subset $\mathrm{J}\subseteq G_4(\RAT_p)$ which maps bijectively to a set of coset representatives for the image of $\mathscr{A}(G_{4,\INTP})$ in $\mathscr{A}(G)^\circ\backslash\mathscr{A}(G_4)/K_{4,p}^\circ$.

Let $\mathrm{E}_4$ be the reflex field of $(G_4,X_4)$. We will fix, as before, a prime $v_4$ of $\mathrm{E}_4$ over $p$, and get $E_4$ with ring of integers $O_{E_4}$. By Lemma \ref{assump--ab type} (3), $v_4$ extends to a place of $\mathrm{E}'\cdot\mathrm{E}_4$, and the completion is $E_4$. By \cite[Lemma 4.6.13]{Paroh}, we have a canonical identification (i.e. a $G_4(\mathbb{A}_f^p)\times \mathrm{Gal}(\overline{E}/E_4)$-equivariant isomorphism)
\[\Sh_{K_{4,p}^\circ}(G_4,X_4)=[[\Sh_{K_p^\circ}(G,X)^+\times \mathscr{A}(G_{4,\INTP})]/\mathscr{A}(G_{\INTP})^\circ]^{|\mathrm{J}|}.\]

Let $\widetilde{\ES}^\mathrm{ad}_{K_p^\circ}(G,X)^+$ be the $G_{\INT_p}^{\mathrm{ad}}$-torsor on $\ES_{K_p^\circ}(G,X)^+$ induced by $\widetilde{\ES}_{K_p}(G,X)$. By \cite[Lemma 4.5.9]{Paroh}, the $\mathscr{A}(G_{\INTP})^\circ$-action on $\ES_{K_p^\circ}(G,X)^+$ lifts to $\widetilde{\ES}^\mathrm{ad}_{K_p^\circ}(G,X)^+$.
One can then consider the quotients
\begin{align*}
\ES_{K_{4,p}^\circ}(G_4,X_4)&:=\big[[\ES_{K_p^\circ}(G,X)^+\times \mathscr{A}(G_{4,\INTP})]/\mathscr{A}(G_{\INTP})^\circ\big]^{|\mathrm{J}|}\\
\text{and\ \ \ \ }\widetilde{\ES}^\mathrm{ad}_{K_{4,p}^\circ}(G_4,K_4)&:=\big[[\widetilde{\ES}^\mathrm{ad}_{K_p^\circ}(G,X)^+\times \mathscr{A}(G_{4,\INTP})]/\mathscr{A}(G_{\INTP})^\circ\big]^{|\mathrm{J}|}
\end{align*}
\begin{theorem}\label{Thm--constuct int ab type}
Notations as above, we have the followings.
\begin{enumerate}
\item {\rm(\cite[Corollary 4.6.15]{Paroh})} $\ES_{K_{4,p}^\circ}(G_4,X_4)$ is represented by an $O_{E_4}$-scheme, and the $G_4(\mathbb{A}_f^p)$-action on $\Sh_{K_{4,p}^\circ}(G_4,X_4)$ extends to $\ES_{K_{4,p}^\circ}(G_4,X_4)$.
\item {\rm(\cite[Corollary 4.6.18]{Paroh})} $\widetilde{\ES}^\mathrm{ad}_{K_{4,p}^\circ}(G_4,K_4)$ is represented by a $G_{\INT_p}^{\mathrm{ad}}$-torsor over $\ES_{K_{4,p}^\circ}(G_4,K_4)$, and the morphism $q:\widetilde{\ES}_{K_{p}}(G,K)\rightarrow \mathrm{M}^\mathrm{loc}_{G,X}$ induces a morphism $q_4:\widetilde{\ES}^\mathrm{ad}_{K_{4,p}^\circ}(G_4,K_4)\rightarrow \mathrm{M}^\mathrm{loc}_{G,X}$ which is smooth and $G_{\INT_p}^{\mathrm{ad}}$-equivariant.
\end{enumerate}
\end{theorem}
\subsection{}\label{conventions} Before finishing this section, we fix some notations and conventions that will be used in the following sections.
\subsubsection{}\label{conven--shv} We will sometimes mention only the level subgroups in notations related to Shimura varieties. For example, for the Shimura datum $(G,X)$ and level $K$, we write $\Sh_K$ for the attached Shimura variety $\Sh_K(G,X)$. Similarly, we have $\ES_K$, $\widetilde{\ES}_K$ $\widetilde{\ES}^\mathrm{ad}_K$; and, if we change $K$ to $K^\circ$, we have $\Sh_{K^\circ}$, $\ES_{K^\circ}$, etc. For the Shimura datum $(G',X')$ (resp. $(G_1,X_1)$) and level $K'$ (resp. $K_1$), the attached objectives will be denoted by $\Sh_{K'}$ (resp. $\Sh_{K_1}$), $\ES_{K'}$ (resp. $\ES_{K_1}$), etc, and similarly if we change $K'$ (resp. $K_1$) to $K'^\circ$ (resp. $K_1^\circ$). Their (local) reflex fields will be denoted by $E$, $E'$ and $E_1$, with rings of integers $O_E$, $O_{E'}$ and $O_{E_1}$ respectively.

\subsubsection{}\label{conven--morp} Let $S\rightarrow S'$ be a morphism, with $X/S$ and $X'/S'$. By a morphism $X\rightarrow X'$, we will \emph{always} assume that the following diagram is commutative.
$$\xymatrix{
X\ar[r]\ar[d]& X'\ar[d]\\
S\ar[r]      & S'}$$

\section[Independence of symplectic embeddings]{Independence of symplectic embeddings}

%\subsection[Some conventions]{Some conventions}

%Independence of symplectic embeddings:

%Notations and assumptions as in \S 2.1, let $s'\in V_{\mathbb{Z}_{(p)}}^\otimes$ be another tensor defining $G_{\INTP}$. We have, by Theorem \ref{result P-Z and K-P}, a tensor $\sdr'\in \V^\otimes$, and that $$\widetilde{\ES}'_K:=\Isom_{\ES_K}\big((V^\vee_{\mathbb{Z}_{(p)}},s')\otimes O_{\ES_K},(\mathcal {V}, \sdr')\big)$$ is a $G_{\INTP}$-torsor. Moreover, the morphism $q':\widetilde{\ES}'_K\rightarrow \mathrm{M}_{G,X}^{\mathrm{loc}}$ induced by $t\mapsto t^{-1}(\V^1)$ is smooth.

%\begin{lemma}
%The $G_{\INTP}$-torsors $\widetilde{\ES}_K$ and $\widetilde{\ES}'_K$ are canonically isomorphic, and we have $q=q'$ via this identification.
%\end{lemma}
%\begin{proof}
%$\widetilde{\ES}_K$ and $\widetilde{\ES}'_K$ are both closed subscheme of $\Isom_{\ES_K}(V_{\INTP}^\vee\otimes O_{\ES_K},\mathcal {V})$. Let $\widetilde{\ES}''_K:=\Isom_{\ES_K}\big((V^\vee_{\mathbb{Z}_{(p)}},s,s')\otimes O_{\ES_K},(\mathcal {V}, \sdr, s'_{\dr})\big),$
%then it is a closed subscheme of both $\widetilde{\ES}_K$ and $\widetilde{\ES}'_K$. But $\widetilde{\ES}''_K$ is a $G_{\INTP}$-torsor over $\ES_K$, and hence $\widetilde{\ES}_K= \widetilde{\ES}''_K=\widetilde{\ES}'_K$. It is clear then that $q=q''=q'$.
%\end{proof}

\subsection{}\label{sum of ebd} Let $i_1:(G,X)\hookrightarrow (\mathrm{GSp}(V_1,\psi_1),S^{\pm}_1)$ and $i_2:(G,X)\hookrightarrow (\mathrm{GSp}(V_2,\psi_2),S^{\pm}_2)$ be two symplectic embeddings. We can construct another symplectic embedding as follows. By the definition of symplectic similitude groups, there is a character $\chi_1:\mathrm{GSp}(V_1,\psi_1)\rightarrow \mathbb{G}_m$, such that $\mathrm{GSp}(V_1,\psi_1)$ acts on $\psi_1$ via $\chi_1$. Note that changing $\chi_1$ to a power of it will not change the symplectic similitude group. Similarly, we have $\chi_2:\mathrm{GSp}(V_2,\psi_2)\rightarrow \mathbb{G}_m$. Let $w:\mathbb{G}_m\rightarrow G$ be the weight cocharacter of $G$. Then $\chi_1\circ w$ and $\chi_2\circ w$ are two characters $\mathbb{G}_m\rightarrow \mathbb{G}_m$ of weights $m_1$ and $m_2$ respectively. After changing $\chi_1$ to $\chi_1^{m_2}$ and $\chi_2$ to $\chi_2^{m_1}$, we see that $G$ acts on $\psi_1$ and $\psi_2$ via the same character. We set $V_3:=V_1\oplus V_2$, and define $\psi_3: V_3\times V_3\rightarrow \mathbb{Q}$,
$$\psi_3\big((v_1,v_2),(v_1',v_2')\big):=\psi_1(v_1,v_1')+\psi_2(v_2,v_2'),\ \ \ \forall\ v_1,v_1'\in V_1 \ \mathrm{and}\ \forall\ v_2,v_2'\in V_2.$$ Then $G\subseteq \mathrm{GSp}(V_3,\psi_3)$, and this embedding induces an embedding of Shimura data $$i_3:(G,X)\hookrightarrow (\mathrm{GSp}(V_3,\psi_3),S_3^{\pm}).$$

\subsection{}\label{comp normal mod} Let $i_1:(G,X)\rightarrow (\mathrm{GSp}(V_1,\psi_1),S^{\pm}_1)$ and $i_2:(G,X)\rightarrow (\mathrm{GSp}(V_2,\psi_2),S^{\pm}_2)$ be as in \S\ref{para int setting} with lattices $V_{t,\INT}\subseteq V_t$, $t=1,2$. By constructions as above, we have a third symplectic embedding $i_3:(G,X)\rightarrow (\mathrm{GSp}(V_3,\psi_3),S^{\pm}_3)$ as well as a lattice $V_{3,\INT}:=V_{1,\INT}\oplus V_{2,\INT}\subseteq V_3$ which also satisfies conditions in \S\ref{para int setting}.

For $t=1,2,3$, we set $d_t=|V_{t,\INT}^\vee/V_{t,\INT}|$ and $g_t=\frac{1}{2}\mathrm{dim}(V_t)$. Clearly, $g_3=g_1+g_2$ and $d_3=d_1d_2$. For $n\geq 3$ an integer such that $(n,p)=1$, let $\mathscr{A}_{g_t,d_t,n}$ be the moduli scheme of abelian schemes over $\mathbb{Z}_{(p)}$-schemes of relative dimension $g_t$ with a polarization $\lambda_t$ of degree $d_t$ and a level $n$ structure $\eta_t$, with $(\mathcal {A}_t,\lambda_t,\eta_t)$ the universal family on it. For $t=1,2$, let $p_t:\mathscr{A}_{g_1,d_1,n}\times \mathscr{A}_{g_2,d_2,n}\rightarrow \mathscr{A}_{g_t,d_t,n}$ be the $t$-th projection. There is also a unique morphism $$j:\mathscr{A}_{g_1,d_1,n}\times \mathscr{A}_{g_2,d_2,n}\longrightarrow \mathscr{A}_{g_3,d_3,n}$$ such that $j^*(\mathcal {A}_3,\lambda_3,\eta_3)=p_1^*(\mathcal {A}_1,\lambda_1,\eta_1)\times p_2^*(\mathcal {A}_2,\lambda_2,\eta_2)$.

Let $K^p\subseteq G(\mathbb{A}_f^p)$ be small enough such that there are morphisms $\Sh_K\rightarrow \mathscr{A}_{g_1,d_1,n}$ and $\Sh_K\rightarrow \mathscr{A}_{g_2,d_2,n}$. We then have a commutative diagram
$$\xymatrix@C=0.6cm{& \mathscr{A}_{g_2,d_2,n}&\\
\Sh_K\ar[rr]^(0.4){i_1\times i_2}\ar[dr]_{i_1}\ar[ur]^{i_2}&& \mathscr{A}_{g_1,d_1,n}\times \mathscr{A}_{g_2,d_2,n}\ar[r]^(0.6){j} \ar[dl]^{p_1}\ar[ul]_{p_2}&\mathscr{A}_{g_3,d_3,n}.\\
& \mathscr{A}_{g_1,d_1,n}&
}\leqno $$
Let $\ES_{1,K}^-$ (resp. $\ES_{2,K}^-$, $\ES_{12,K}^-$, $\ES_{3,K}^-$) be the scheme theoretic image of $i_1$ (resp. $i_2$, $i_1\times i_2$, $i_3:=j\circ(i_1\times i_2)$), and $\ES_{1,K}$ (resp. $\ES_{2,K}$, $\ES_{12,K}$, $\ES_{3,K}$) be its normalization. We have the following induced diagram of integral models of $\Sh_K$
$$\xymatrix{\ES_{1,K}&\ES_{12,K}\ar[d]^{j}\ar[l]_{p_1}\ar[r]^{p_2}& \ES_{2,K}\\
&\ES_{3,K}.
}$$

For $t=1,2,3$, $\ES_{t,K}$ is the integral model for $\Sh_K$ as in \S\ref{para int setting} induced by the embedding $i_t$.
\begin{lemma}\label{i' isom}
The morphism $j:\ES_{12,K}\rightarrow \ES_{3,K}$ is an isomorphism.
\end{lemma}
\begin{proof} Notations as above, we have, on $\ES_{3,K}$, a polarized abelian scheme with level $n$-structure $(\mathcal {A}_3, \lambda_3, \eta_3)|_{\ES_{3,K}}$. By construction, $\mathcal {A}_3|_{\Sh_K}=\mathcal {A}_1|_{\Sh_K}\times \mathcal {A}_2|_{\Sh_K}$. Let $\pi_t$, $t=1,2$, be the endomorphism of $\mathcal {A}_3|_{\Sh_K}$ which is the identity on $\mathcal {A}_t|_{\Sh_K}$ and zero on $\mathcal {A}_{3-t}|_{\Sh_K}$, then by \cite{deg of abv} Chapter 1, Proposition 2.7, $\pi_t$ extends uniquely to $\ES_{3,K}$ and hence $\mathcal {A}_3|_{\ES_{3,K}}=\mathcal {A}_1|_{\ES_{3,K}}\times \mathcal {A}_2|_{\ES_{3,K}}$. But then on $\ES_{3,K}$, we have $\lambda_3=\lambda_1\times \lambda_2$, and $\eta_3=\eta_1\times \eta_2$, as it is so on $\Sh_K$. In particular, $(\mathcal {A}_3, \lambda_3, \eta_3)|_{\ES_{3,K}}$ induces a morphism to $\mathscr{A}_{g_1,d_1,n}\times \mathscr{A}_{g_2,d_2,n}$, and hence a section of $j$. Then $j$ is an isomorphism, as $\ES_{12,K}$ and $\ES_{3,K}$ has the same generic fiber.
\end{proof}

We will identify $\ES_{12,K}$ and $\ES_{3,K}$ from now on. Furthermore, by taking normalizations in $\Sh_{K^\circ}$, we can replace $K$ by $K^\circ$ in the above diagram. Let $s_t\in V_{t,\mathbb{Z}_{(p)}}^\otimes$ be a tensor defining $G_{\INTP}$, it induces a tensor $s_{t,\mathrm{dR}}\in\V_t^\otimes$, where $\V_t:=\Hdr(\mathcal {A}_t/\ES_{t,K})$. Let $\V^1_t\subseteq \V_t$ the Hodge filtration; and $\widetilde{\ES}_{t,K}$ be the corresponding $G_{\INTP}$-torsor on $\ES_{t,K}$ as in Theorem \ref{result P-Z and K-P} (2). We then have diagrams $\ES_{t,K}\stackrel{\pi_t}{\longleftarrow}\widetilde{\ES}_{t,K}\stackrel{q_t}{\longrightarrow} \mathrm{M}^\mathrm{loc}_{G,X}$ as in Theorem \ref{result P-Z and K-P} (3).

Let
$\widetilde{\ES}'_{3,K}:=\Isom_{\ES_{3,K}}\big((V^\vee_{3,\INT_p}\supseteq V^\vee_{1,\INT_p},s_3, s_1)\otimes O_{\ES_{3,K}},\ (\V_3\supseteq\V_1, s_{3,\dr},s_{1,\dr})\big).$

\begin{lemma}\label{I is torsor}
The inclusion $\widetilde{\ES}'_{3,K}\subseteq \widetilde{\ES}_{3,K}$ is an isomorphism.
\end{lemma}
\begin{proof}
It suffices to check that $\V_1=V^\vee_{1,\INT_p}\times^{G_{\INT_p}} \widetilde{\ES}_{3,K}$, and that $s_{1,\dr}:O_{\ES_{3,K}}\rightarrow \V_1^\otimes$ is the same as $(\INT_p\stackrel{s_1}{\rightarrow} V^\otimes_{1,\INT_p})\times^{G_{\INT_p}} \widetilde{\ES}_{3,K}$. But we can pass to the generic fiber first and then to $\mathbb{C}$. Now the statements become standard facts, see e.g. \cite[\S2.2]{CIMK}.
\end{proof}
\begin{theorem}\label{uniq int model}
For $t=1,2$, the morphism $p_t:\ES_{3,K}\rightarrow \ES_{t,K}$ is an isomorphism. Moreover, the same statement holds if we change $K$ to $K^\circ$.
\end{theorem}
\begin{proof}
The theorem follows by putting together the following claims which will be proved separately.

\textbf{Claim 1: } the morphism $p_1:\ES_{3,K}\rightarrow \ES_{1,K}$ induces isomorphisms on stalks of closed points in special fibers. By the previous lemma, the assignment $f\mapsto f|_{V_{1,\INT_p}^\vee}$ induces an isomorphism $\widetilde{\ES}_{3,K}\rightarrow p_1^*\widetilde{\ES}_{1,K}$ of $G_{\INT_p}$-torsors over $\ES_{3,K}$, and we have the following commutative diagram.
$$\xymatrix{\ES_{3,K}\ar[d]_{p_1} \ar@{}[dr]|{\Box}& \widetilde{\ES}_{3,K}\ar[r]^{q_3}\ar[l]_{\pi_3}\ar[d]& \mathrm{M}_{G,X}^{\mathrm{loc}}\ar@{=}[d]\\
\ES_{1,K} & \widetilde{\ES}_{1,K} \ar[r]^{q_1}\ar[l]_{\pi_1}& \mathrm{M}_{G,X}^{\mathrm{loc}}
}$$

Let $E^p$ be the maximal unramified extension of $E$, we will base-change to $O_{E^p}$ without changing notations. For $x\in \ES_{3,K}(\overline{\kappa})$ and $y\in \ES_{1,K}(\overline{\kappa})$ with $p_1(x)=y$, we write $R_x$ and $R_y$ for the completion of stalks respectively. For a section $f\in \widetilde{\ES}_{1,K}(R_y)$, we denote by $z\in \mathrm{M}_{G,X}^{\mathrm{loc}})(\overline{\kappa})$ the image of $f_y$ and $R_z$ the completion of stalk. The filtration $f^{-1}(\V^1_1)\subseteq V^\vee_{1,\INTP}\otimes R_y$ induces an isomorphism $R_z\rightarrow R_y$ by Theorem \ref{result P-Z and K-P} (3). Similarly, $p_1^*(f)\in \widetilde{\ES}_{3,K}(R_x)$ induces an isomorphism $R_z\rightarrow R_x$ with the factorization $R_z\rightarrow R_y\stackrel{p_1}{\rightarrow} R_x$. In particular, $p_1$ induces isomorphisms on stalks.

\textbf{Claim 2: }$p_1:\ES_{3,K}\rightarrow \ES_{1,K}$ is surjective on special fibers. An $\overline{\mathbb{F}}_p$-point of $\ES_{1,K}$ factors through an $A$-point of $\ES_{1,K}$ with $A$ a strictly henselian DVR of mixed characteristic. It lifts to $\ES_{1,K_p}:=\varprojlim_{K^p}\ES_{1,K_pK^p}$ as the transition morphisms are \'{e}tale, and induces an $A[1/p]$-point of $\ES_{3,K_p}[1/p]=\ES_{1,K_p}[1/p]$. By the N\'{e}ron-Ogg-Shafarevich criterion, it extends to an $A$-point of $\ES_{3,K_p}$.

To sum up, $p_1:\ES_{3,K}\rightarrow \ES_{1,K}$ is a surjective birational morphism between normal schemes which induces isomorphisms on completions of stalks of closed points, and hence is an isomorphism, same for $p_2$. These statements hold if we replace $K$ by $K^\circ$, as $\ES_{K^\circ}$ is by definition the normalization of the closure of the image of $\Sh_{K^\circ}$ in $\ES_{K}$.
\end{proof}

\begin{remark}\label{remark--indp of hodge typ loc diag}Let $\widetilde{\ES}_{K^\circ}$ be the pull back to $\ES_{K^\circ}$ of $\widetilde{\ES}_{K}$. We actually show, in the above proof, that the diagrams $$\ES_{K}\stackrel{\pi}{\longleftarrow} \widetilde{\ES}_{K}\stackrel{q}{\longrightarrow} \mathrm{M}_{G,X}^{\mathrm{loc}}\text{\ \ \ \ and\ \ \ \ }\ES_{K^\circ}\stackrel{\pi}{\longleftarrow} \widetilde{\ES}_{K^\circ}\stackrel{q}{\longrightarrow} \mathrm{M}_{G,X}^{\mathrm{loc}}$$ are independent of choices of symplectic embeddings.
\end{remark}
\subsection{}\label{setting--mor Hodge type} One can use the previous discussions to show that, under certain conditions, morphisms of Shimura varieties of Hodge type extend to integral models. Let $f:(G,X)\rightarrow (G',X')$ be a morphism of Shimura data of Hodge type. We assume in addition that $G_{\mathbb{Q}_p}\rightarrow G'_{\mathbb{Q}_p}$ extends to a homomorphism $G_{\INTP}\rightarrow G'_{\INTP}$ of group schemes over $\INTP$. For $K=K_pK^p\subseteq G(\mathbb{A}_f)$ and $K'=K'_pK'^p\subseteq G'(\mathbb{A}_f)$ with $f(K)\subseteq K'$, we have a natural morphism $\Sh_K\rightarrow \Sh_{K',E}.$
Noting that $f(K^\circ)\subseteq K'^\circ$ for $K^\circ:=K^\circ_pK^p$ and $K'^\circ:=K'^\circ_pK'^p$, we also have a natural morphism $\Sh_{K^\circ}\rightarrow \Sh_{K'^\circ,E}.$

Choosing symplectic embeddings $$i:(G,X)\rightarrow (\mathrm{GSp}(V,\psi),S^{\pm})\text{\ \ and\ \ }i':(G',X')\rightarrow (\mathrm{GSp}(V',\psi'),S'^{\pm})$$ with $V_\INT\subseteq V$ and $V'_\INT\subseteq V'$ as in \S\ref{para int setting}, we have integral models $\ES_{K}$, $\ES_{K^\circ}$, $\ES_{K'}$, and $\ES_{K'^\circ}$ for $\Sh_{K}$, $\Sh_{K^\circ}$, $\Sh_{K'}$, and $\Sh_{K'^\circ}$ respectively.

\begin{proposition}\label{extn morp on int}
The morphism\footnote{Recall that our conventions are as in (\ref{conven--morp}).} of Shimura varieties $\Sh_{K}\rightarrow \Sh_{K'}$ extends to a morphism of integral models $\ES_{K}\rightarrow \ES_{K'}$. The same statement holds if we change $K$ to $K^\circ$ and $K'$ to $K'^\circ$.
\end{proposition}
\begin{proof}
It suffices to show the statement for $K$ and $K'$. Notations as above, we set $V_1:=V\oplus V'$, and define $\psi_1: V_1\times V_1\rightarrow \mathbb{Q}$ by
\begin{subeqn}\label{symp for morph of hodge}
\psi_1\big((v,v'),(w,w')\big):=\psi(v,w)+\psi'(v',w'),\ \ \forall\ v,w\in V \ \mathrm{and}\ \forall\ w,w'\in V'.
\end{subeqn}
Then as in \S\ref{sum of ebd}, we have a symplectic embedding $i_1:(G,X)\subseteq (\mathrm{GSp}(V_1,\psi_1),S^{\pm}_1)$, inducing a commutative diagram
$$\xymatrix@C=1.3cm{
\Sh_K\ar[r]_{i\times(i'\circ f)}\ar[d]_f \ar@/^0.7pc/[rr]^{i_1}&\mathscr{A}_{g,d,n}\times \mathscr{A}_{g',d',n}\ar[r]\ar[d]_{p_2}&\mathscr{A}_{g_1,d_1,n}\\
\Sh_{K'}\ar[r]_{i'}&\mathscr{A}_{g',d',n}.}$$
Noting that for $V_{1,\INT}:=V_\INT\oplus V'_{\INT}$, we have $G_{\INT_p}\subseteq \mathrm{GL}(V_{1,\mathbb{Z}_p})$, and by Theorem \ref{uniq int model}, $\ES_{K}$ is the normalization of the closure of the image of $\Sh_K$ in $\mathscr{A}_{g_1,d_1,n}$. But by the proof of Lemma \ref{i' isom}, $\ES^0_K$, the normalization of the closure of the image of $\Sh_K$ in $\mathscr{A}_{g,d,n}\times \mathscr{A}_{g',d',n'}$, is isomorphic to $\ES_{K}$, and hence $f$ extend to a morphism of integral models.
\end{proof}
\subsection{} The local model $\mathrm{M}_{G,X}^{\mathrm{loc}}$ is independent of choices of symplectic embeddings, and hence we have a morphism $q_f: \mathrm{M}_{G,X}^{\mathrm{loc}}\rightarrow \mathrm{M}_{G',X'}^{\mathrm{loc}}$. Let $\ES_{K'}\stackrel{\pi_{K'}}{\longleftarrow} \widetilde{\ES}_{K'}\stackrel{q_{K'}}{\longrightarrow} \mathrm{M}_{G',X'}^{\mathrm{loc}}$ be the diagram attached to the embedding $i'$ as in Theorem \ref{result P-Z and K-P}. The morphism $f:\ES_K\rightarrow \ES_{K'}$ obtained as above sits inside the following commutative diagram.
\subsubsection{}\label{functo local mod diag} \ \ \ \ \ \ \ \ \ $\xymatrix@C=1.3cm@R=0.5cm{
 &\widetilde{\ES}_K\ar[dl]_{\pi_K}\ar[dr]^{q_K}\ar[d]   \\
\ES_{K}\ar@{}[ddr] |{\Box} \ar[d]_{f}  & \widetilde{\ES}_K\times^{G_{\INT_p}}G'_{\INT_p}\ar[l]\ar[d]^(0.45){\cong}& \mathrm{M}_{G,X}^{\mathrm{loc}}\ar[d]^{q_f}\\
\ES_{K'} & f^*\widetilde{\ES}_{K'}\ar[ul]\ar[d]\ar[r] & \mathrm{M}_{G',X'}^{\mathrm{loc}}\\
   &    \widetilde{\ES}_{K'} \ar[ur]_{q_{K'}}\ar[ul]^{\pi_{K'}}       }$\\
Here $\widetilde{\ES}_K\rightarrow f^*\widetilde{\ES}_{K'}$ is induced by $\alpha\mapsto \alpha|_{V'^\vee_{\INT_p}}$, where $\alpha$ is a section of $\widetilde{\ES}_K$ constructed via the symplectic embedding $i_1$ (see Remark \ref{remark--indp of hodge typ loc diag}).

\section[A variation: connected Shimura varieties of Hodge type]{A variation: connected Shimura varieties of Hodge type}
\subsection{} A \emph{connected Shimura datum} is a pair $(H,Y)$ where $H$ is a semi-simple algebraic group over $\mathbb{Q}$, and $Y$ is a connected component of $X$, with $(H^{\text{ad}}, X)$ an adjoint Shimura datum. There is a connected Shimura variety $\Sh^+(H,Y)_{\mathbb{C}}$ attached to $(H,Y)$, which is defined over $\overline{\RAT}$. Given a Shimura datum $(G,X)$ with a connected component $X^+\subseteq X$, we get a connected Shimura datum $(G^{\text{der}}, X^+)$. The connected Shimura variety $\Sh^+(G^{\text{der}}, X^+)_{\mathbb{C}}$ coincides with the connected component of $\Sh(G,X)_{\mathbb{C}}$ containing the image of $X^+\times \{\mathrm{id}\}$.

A connected Shimura datum $(H,Y)$ is said to be \emph{of Hodge type}, if there is a Shimura datum of Hodge type $(G,X)$, with $H=G^{\text{der}}$ and $Y$ a  connected component of $X$. The Shimura datum $(G,X)$ will be called an \emph{extension} of $(H,Y)$.

\subsection{}\label{morp of conn data--setting} Let $(H,Y)$ and $(H',Y')$ be two connected Shimura data of Hodge type. We will consider homomorphisms $f:H\rightarrow H'$ such that there exists a morphism of Shimura data $f_4:(G_4,X_4)\rightarrow(G'_4,X'_4)$ sitting inside a commutative diagram $$\xymatrix{
H\ar[r]^{f}\ar[d]^{\alpha}&H'\ar[d]^{\alpha'}\\
G_4\ar[r]^{f_4}&G'_4}$$
satisfying the following conditions:
\begin{enumerate}
\item $\alpha$ is a central isogenies inducing an isomorphism of adjoint Shimura data $(H^{\mathrm{ad}},Y^{\mathrm{ad}})\stackrel{\cong}{\longrightarrow} (G_4^{\mathrm{ad}},X_4^{\mathrm{ad}})$, and same for $\alpha'$;
\item via the above identifications, $Y\subseteq X_4$, $Y'\subseteq X'_4$ and $f_4(Y)\subseteq Y'$.
\end{enumerate}

\subsection{}\label{morp of conn to symp ebd} Notations and assumptions as above, we choose a Hodge type extension $(G,X)$ (resp. $(G',X')$) for $(H,Y)$ (resp. $(H',Y')$), together with a symplectic embedding $i:(G,X)\hookrightarrow (\mathrm{GSp}(V,\psi), S^{\pm})$ (resp. $i':(G',X')\hookrightarrow (\mathrm{GSp}(V',\psi'), S'^{\pm})$). We will explain how to ``compare'' these two embeddings.

Let $Z$ be the center of $G$, we set $Z^{\mathrm{Sp}}$ (resp. $G^{\mathrm{Sp}}$) to be the \emph{identity component} of $Z\cap \mathrm{Sp}(V,\psi)$ (resp. $G\cap \mathrm{Sp}(V,\psi)$). The multiplication $H\times Z^{\mathrm{Sp}}\rightarrow G^{\mathrm{Sp}}$ induces an identity $G^{\mathrm{Sp}}=H\cdot Z^{\mathrm{Sp}}$. Let $\mathbb{S}^1\subseteq \mathbb{S}$ be the unit circle, for $h:\mathbb{S}\rightarrow G_{\mathbb{R}}$ in $Y$, we identify $h$ with its restriction to $\mathbb{S}^1$ (still denoted by $h$), and its image in $X_4$ (resp. $Y^{\mathrm{ad}}$) will be denoted by $h_4$ (resp. $\overline{h}$). Now $h:\mathbb{S}^1\rightarrow G_{\mathbb{R}}$ factors through $G^{\mathrm{Sp}}_{\mathbb{R}}$, and it decomposes as $h=\widetilde{h}_4\cdot h_Z$, where $\widetilde{h}_4:\mathbb{G}_{m,\mathbb{C}}\rightarrow H_{\mathbb{C}}$ is a fractional cocharacter lifting $h_4$, and $h_Z$ is a fractional cocharacter of $Z^{\mathrm{Sp}}_{\mathbb{C}}$.

Using similar notations, we have $Z'^{\mathrm{Sp}}$ and $G'^{\mathrm{Sp}}$ with $G'^{\mathrm{Sp}}=H'\cdot Z'^{\mathrm{Sp}}$; and for $h':\mathbb{S}^1\rightarrow G'^{\mathrm{Sp}}_{\mathbb{R}}$ in $Y'$ such that $h'_4=f\circ h_4$, we have $h'=(f\circ\widetilde{h}_4)\cdot h'_{Z'}$, where $h'_{Z'}$ is a fractional cocharacter of $Z'^{\mathrm{Sp}}_{\mathbb{C}}$.

%Let $i:(G,X)\hookrightarrow (\mathrm{GSp}(V,\psi), S^{\pm})$ be a symplectic embedding, and $G_1$ be the subgroup generated by $H$ and the diagonal torus $\mathbb{G}_m\subseteq \mathrm{GL}(V)$. Then $G_1\subseteq G$, and by the second paragraph of the proof to \cite[Lemma 1.3.5]{varideshi}, it extends to an embedding of Hodge type Shimura data $(G_1,X_1)\hookrightarrow(G,X)$.

%Let  $(H',Y')$ be another connected Shimura datum of Hodge type, with an extension $(G',X')$ and a symplectic embedding $i':(G',X')\hookrightarrow (\mathrm{GSp}(V',\psi'), S'^{\pm})$. Let $f:(H,Y)\rightarrow (H', Y')$ be a morphism of connected Shimura data; and $(V_1,\psi_1)$ be as in (\ref{symp for morph of hodge}). We will construct a new extension of $(H,Y)$ which embeds into $(\mathrm{GSp}(V_1,\psi_1), S_1^{\pm}))$ and ``compares'' $(H,Y)$ with $(H',Y')$.

Let $V_1=V\oplus V'$, and $G_1$ be the subgroup of $\mathrm{GL}(V_1)$ generated by $Z^{\mathrm{Sp}}\times Z'^{\mathrm{Sp}}$, the graph of $f:H\rightarrow H'$ (as a subgroup of $\mathrm{GL}(V)\times \mathrm{GL}(V')$) and the diagonal torus $\mathbb{G}_m\subseteq \mathrm{GL}(V_1)$. It is by construction that $G_1\subseteq G\times G'$ and $G_1^{\mathrm{der}}=H$. Moreover, by our discussions above, the homomorphism $h_1=(h,h'):\mathbb{S}\rightarrow G_{\mathbb{R}}\times G'_\mathbb{R}$ factors through $G_{1,\mathbb{R}}$. Let $X_1$ be the $G_1(\mathbb{R})$-orbit of $h_1$, then $(G_1,X_1)$ is a Hodge type extension of $(H,Y)$ with a symplectic embedding $i_1:(G_1,X_1)\hookrightarrow (\mathrm{GSp}(V_1,\psi_1), S_1^{\pm})$, and it is equipped with morphisms $(G,X)\leftarrow (G_1,X_1)\rightarrow (G',X')$.

\subsection{}\label{extend morp gp sch} Let $H_{\INT_p}$ (resp. $H'_{\INT_p}$) be the Bruhat-Tits group scheme attached to a point $\mathfrak{b}$ (resp. $\mathfrak{b}'$) in the Bruhat-Tits building $\mathcal {B}(H,\mathbb{Q}_p)$ (resp. $\mathcal {B}(H',\mathbb{Q}_p)$).
We assume in addition that
\subsubsection{}\label{assump--ext derived} \ \ \ \ \ $f_{\RAT_p}:H_{\RAT_p}\rightarrow H'_{\RAT_p}$ extends to a homomorphism $H_{\INT_p}\rightarrow H'_{\INT_p}$.\\
Notations as above and assume that $i,i'$ are as in \S\ref{para int setting} with lattices $V_{\INT_p}\subseteq V$ and $V'_{\INT_p}\subseteq V'$ also as in loc.cit. In particular, $G_{\INT_p}$ (resp. $G'_{\INT_p}$), the closure of $G_{\RAT_p}$ (resp. $G'_{\RAT_p}$) in $\mathrm{GL}(V_{\INT_p})$ (resp. $\mathrm{GL}(V'_{\INT_p})$), is the Bruhat-Tits group scheme attached to $\mathfrak{b}$ (resp. $\mathfrak{b}'$). Noting that $Z_{\RAT_p}^{\mathrm{Sp}}$ (resp. $Z'^{\mathrm{Sp}}_{\RAT_p}$) is connected by definition, its closure $Z^{\mathrm{Sp}}_{\INT_p}$ (resp. $Z'^{\mathrm{Sp}}_{\INT_p}$) in $G_{\INT_p}$ (resp. $G'_{\INT_p}$) is smooth by \cite[Remark 1.1.8]{Paroh}.

Now we consider the commutative diagram
$$\xymatrix@C=1.3cm{
G_{1,\RAT_p}\ar[r]\ar[d]& \mathrm{GL}(V_{\INT_p})\times\mathrm{GL}(V'_{\INT_p})\ar@{^{(}->}[r]\ar[d]^{p_1}& \mathrm{GL}(V_{1,\INT_p})\\
G_{\RAT_p}\ar[r]& \mathrm{GL}(V_{\INT_p}).}$$
Let $G_{1,\INT_p}$ be the closure of $G_{1,\RAT_p}$ in $\mathrm{GL}(V_{1,\INT_p})$. By our assumption (\ref{assump--ext derived}), the closure of $H_{\RAT_p}$ in $\mathrm{GL}(V_{1,\INT_p})$ is $H_{\INT_p}$, and hence $G_{1,\INT_p}=H_{\INT_p}\cdot (Z^{\mathrm{Sp}}_{\INT_p}\times Z'^{\mathrm{Sp}}_{\INT_p})\cdot \mathbb{G}_m$ is necessarily smooth. In particular, it is the Bruhat-Tits group scheme attached to $\mathfrak{b}$ via the inclusion $\mathcal{B}(H,\mathbb{Q}_p)\subseteq \mathcal{B}(G_1,\mathbb{Q}_p)$. The projection $p_1$ induces homomorphisms $G_{1,\INT_p}\rightarrow G_{\INT_p}$ and $G^\circ_{1,\INT_p}\rightarrow G^\circ_{\INT_p}$ extending $G_{1,\RAT_p}\rightarrow G_{\RAT_p}$. As we have remarked at the beginning of \S\ref{sec--int mod}, they are defined over $\INT_{(p)}$. Similarly, $G_1\rightarrow G'$ extends to homomorphisms $G_{1,\INTP}\rightarrow G'_{\INTP}$ and $G^\circ_{1,\INTP}\rightarrow G'^\circ_{\INTP}$.

%We remark that by \cite[\S1.1.2]{Paroh}, $G^\circ_{1,\INT_p}$ is a parahoric group scheme.

%\begin{example}
%Let $f:(G_1,X_1)\rightarrow (G_2,X_2)$ be a morphism of Shimura data. For $i=1,2$, let $H_i$ be a semi-simple group with a central isogeny $H_i\rightarrow G_i^{\mathrm{der}}$ which induces an isomorphism of adjoint Shimura data. We assume, in addition, that there is a homomorphism $\widetilde{f}:H_1\rightarrow H_2$ lifting $f$.  be a central isogeny of semi-simple groups inducing an isomorphism $f^{\mathrm{ad}}: (H^{\mathrm{ad}},X^{\mathrm{ad}})\rightarrow (H'^{\mathrm{ad}},X'^{\mathrm{ad}})$ of Shimura data. Given connected components $Y\subseteq X^{\mathrm{ad}}$ and $Y'\subseteq X^{\mathrm{ad}}$, with a $g_0\in H'^{\mathrm{ad}}(\RAT)$ such that $f^{\mathrm{ad}}(Y)=g_0\cdot Y'$, we have a morphism of connected Shimura data $(H,Y)\rightarrow (H',Y')$ induced by $$g\mapsto g_0f(g)g_0^{-1} \ \ \ \text{and \ \ \ }y\mapsto g_0f^{\mathrm{ad}}(y)g_0^{-1}, \ \ \ \forall g\in H, \ \forall y\in Y.$$
%\end{example}
%\begin{example}
%More generally, let $f:H\rightarrow H'$ be a homomorphism of semi-simple groups respecting centers, inducing a morphism $f^{\mathrm{ad}}: (H^{\mathrm{ad}},X^{\mathrm{ad}})\rightarrow (H'^{\mathrm{ad}},X'^{\mathrm{ad}})$ of adjoint Shimura data. Let $Y,Y',g_0$ be as above with $f^{\mathrm{ad}}(Y)\subseteq g_0\cdot Y'$, we get a morphism of connected Shimura data $(H,Y)\rightarrow (H',Y')$ exactly in the same way.
%\end{example}

\subsection{} We will work with connected components of $\Sh_{K_p^\circ}(G,X)_{\mathbb{C}}$. The connected component containing the image of $X^+\times \{\mathrm{id}\}$ is $\Sh^+_{K_p^\circ}(G,X)_{\mathbb{C}}:=\varprojlim_{K^p} \Gamma_{K^p}\backslash X^+$, where $\Gamma_{K^p}$ is the image of $G(\RAT)_+\cap K_p^\circ K^p$ in $G^{\mathrm{ad}}(\RAT)$. By \cite[Corollary 4.3.9]{Paroh}, it is defined over $\mathrm{E}^p$, the maximal extension of $\mathrm{E}$ which is unramified over primes dividing $p$.

Let $\Gamma_p$ be the intersection of the $\Gamma_{K^p}$s. Then $$\Gamma_p=\bigcap_L\mathrm{Im}\big((G^{\mathrm{der}\circ}(\INTP)_+\cap L)\rightarrow G^{\mathrm{ad}}(\RAT)\big)$$ where $L$ runs through $p$-congruence subgroups of $G^{\mathrm{der}}(\RAT)$, and hence depends only on $G^{\mathrm{der}\circ}_{\INTP}$. The $\mathrm{E}^p$-variety $\Sh^+_{K_p^\circ}(G,X)_{\mathrm{E}^p}$, or simply $\Sh^+_{K_p^\circ, \mathrm{E}^p}$, depends only on $(G^{\mathrm{der}\circ}_{\INTP},X^+)$, and hence will be denoted by $\Sh_{\Gamma_p}(G^{\mathrm{der}}, X^+)_{\mathrm{E}^p}$, or simply $\Sh^+_{\Gamma_p,\mathrm{E}^p}$ sometimes. We will always pass to $E^p$, and denoted the variety by $\Sh^+_{K_p^\circ}$ or $\Sh^+_{\Gamma_p}$, if there is no risk of confusion.

%By \cite[Corollary 2.0.13]{varideshi}, $\Gamma_p$ is the intersection of images of groups of the form $G^{\circ}(\INTP)_+\cap (G^{\mathrm{der}}(\RAT)\cap K^p)\cdot U$, where $U$ is a finite index subgroup of the group of $p$-units in $Z^0_G(\RAT)$, whose image, by the proof of \cite[Lemma 4.6.4]{Paroh}, coincides with that of $G^{\mathrm{der}\circ}(\INTP)_+\cap K^p$. In particular,

%$G^{\mathrm{der}\circ}(\INTP)_+\cap K^p$ Let $\mathscr{A}(G_{\INTP})^\circ$ be as in ???, by \cite[Lemma 4.6.4 (2)]{Paroh}, $\Sh^+_{K_p^\circ}(G,X)_{\mathbb{C}}=\varprojlim \Gamma\backslash X^+$, where $\Gamma$ runs through open subgroups in $\mathscr{A}(G_{\INTP})^\circ$. In particular, for any $E_1$ with $\mathrm{E}^p \subseteq E_1\subseteq \overline{\RAT}$, the $E_1$-scheme $\Sh^+_{K_p^\circ}(G,X)$ depends only on $G^{\mathrm{der}\circ}_{\INTP}$ and $X^+$, and hence will be denoted by $\Sh_{\Gamma_p}(G^{\mathrm{der}},X^+)$,

Let $(G',X')$ be another Hodge type Shimura datum with reflex field $\mathrm{E}'$, and $f:H=G^{\mathrm{der}}\rightarrow H'=G'^{\mathrm{der}}$ be a homomorphism as in \S\ref{morp of conn data--setting}. We assume in addition that condition (\ref{assump--ext derived}) holds. By the proof of \cite[Corollary 5.4]{tradeshi}, the morphism of complex varieties $f:\Sh^+_{\Gamma_p,\mathbb{C}}\rightarrow \Sh^+_{\Gamma'_p, \mathbb{C}}$ is defined over $\mathrm{E}^p\cdot \mathrm{E}'^p$. As what we used to do, we will fix a prime of $\mathrm{E}^p\cdot \mathrm{E}'^p$ over $v$ and pass to completion. It will be denoted by $\mathtt{E}$ with ring of integers $\mathtt{O}$.

\begin{theorem}\label{result for conn shv}Notations and assumptions as above,
\begin{enumerate}
\item the integral model $\ES_{\Gamma_p}^+$ is independent of choices of extensions to Shimura data of Hodge type;
\item %let $f:(H,Y)\rightarrow (H', Y')$ be a morphism of connected Shimura data satisfying ????,
$f:\Sh^+_{\Gamma_p, \mathtt{E}}\rightarrow \Sh^+_{\Gamma'_p,\mathtt{E}}$ extends to a morphism $\ES^+_{\Gamma_p,\mathtt{O}}\rightarrow \ES^+_{\Gamma'_p,\mathtt{O}}$.
\end{enumerate}
\end{theorem}
\begin{proof}
Before starting the proof, we would like to point out that the proof for (1) and (2) follows slightly different notation systems---for (1), we follow \S\ref{sum of ebd}; while for (2), we switch to \S\ref{setting--mor Hodge type}.

We prove (1) first. Let $$i_1:(G_1,X_1)\rightarrow (\mathrm{GSp}(V_1,\psi_1),S^{\pm}_1)\text{ \ \ and \ \ }i_2:(G_2,X_2)\rightarrow (\mathrm{GSp}(V_2,\psi_2),S^{\pm}_2)$$ be extensions of $(H,Y)$ with symplectic embedding; and  $V_{t,\INT}\subseteq V_t$, $t=1,2$, be lattices as in \S\ref{para int setting}. We have then $(\mathrm{GSp}(V_3,\psi_3),S^{\pm}_3)$ with $V_{3,\INT}:=V_{1,\INT}\oplus V_{2,\INT}\subseteq V_3$ as at the beginning of  \S\ref{sum of ebd}.

By \S\ref{morp of conn to symp ebd}, we have a new extension $(G_3,X_3)$ of $(H,Y)$ with a symplectic embedding $(G_3,X_3)\rightarrow (\mathrm{GSp}(V_3,\psi_3),S^{\pm}_3)$. For $t=1,2$, we have natural morphisms $p_t:(G_3,X_3)\rightarrow (G_t,X_t)$, which induce, using similar notations as in \S\ref{conventions}, morphisms of Shimura varieties $\Sh_{K_{3,p}}\rightarrow \Sh_{K_{t,p}}$.

Let $E$ be the $p$-adic reflex field of $(G_3,X_3)$\footnote{Of course, one can also use the $p$-adic reflex field of $(G_1,X_1)$ or $(G_2,X_2)$} with ring of integers $O$. By \S\ref{extend morp gp sch} (take $f=\mathrm{id}_H$, $V=V_1$, and $V'=V_2$), the homomorphisms $G_{3}\rightarrow G_{t}$ extend to homomorphisms of group schemes $G_{3,\INTP}\rightarrow G_{t,\INTP}$, and hence by Proposition \ref{extn morp on int}, the morphisms $\Sh_{K_{3,p},E^p}\rightarrow \Sh_{K_{t,p},E^p}$ extends to $\ES_{K_{3,p},O^p}\rightarrow \ES_{K_{t,p},O^p}$. Restricting to components containing the image of $X^+\times\{\mathrm{id}\}$, we have morphisms $\ES^+_{K_{3,p},O^p}\rightarrow \ES^+_{K_{t,p},O^p}$ which are generically isomorphisms.

We claim that for $t=1,2$, the morphisms $\ES^+_{K_{3,p},O^p}\rightarrow \ES^+_{K_{t,p},O^p}$ are isomorphisms, and hence proves (1). By \cite[Proposition 2.2.7]{Paroh}, we have natural isomorphisms $\mathrm{M}_{G_3,X_3}^{\mathrm{loc}}\rightarrow \mathrm{M}_{G_t,X_t}^{\mathrm{loc}}$; and by (\ref{functo local mod diag}), we have the following commutative diagram.

\subsubsection{}\label{diag--conn loc mod diag} \ \ \ \ \ $\xymatrix@C=1.8cm@R=0.7cm{
\ES^+_{K_{3,p},O^p}\ar@{}[ddr] |{\Box} \ar[d]_{f}  & \widetilde{\ES}^+_{K_{3,p},O^p}\ar[d]\ar[l]_{\pi_K}\ar[r]^{q_K}& \mathrm{M}_{G_3,X_3}^{\mathrm{loc}}\ar[d]^{\cong}\\
\ES^+_{K_{t,p},O^p} & \widetilde{\ES}^+_{K_{3,p},O^p}\times^{G_{3,\INT_p}}G_{t,\INT_p}\ar[ul]\ar[d]\ar[r] & \mathrm{M}_{G_t,X_t}^{\mathrm{loc}}\\
   &    \widetilde{\ES}^+_{K_{t,p},O^p} \ar[ur]_{q_{K'}}\ar[ul]^{\pi_{K'}}       }$\\
In particular, $\ES^+_{K_{3,p},O^p}\rightarrow \ES^+_{K_{t,p},O^p}$ induces an isomorphism of complete local rings at closed points. But then by the same argument as in the proof of Theorem \ref{uniq int model}, it is an isomorphism.

Now we come to (2). The proof is essentially a variation of that of Proposition \ref{extn morp on int}. Notations as in \S\ref{morp of conn to symp ebd}, let $V_{\INT}\subseteq V$ and $V'_{\INT}\subseteq V'$ be lattices as in \S\ref{para int setting}. For $(\mathrm{GSp}(V_1,\psi_1),S^{\pm}_1)$ with $V_{1,\INT}:=V_{\INT}\oplus V'_{\INT}\subseteq V_1$, we have, by \S\ref{morp of conn to symp ebd}, a new extension $(G_1,X_1)$ of $(H,Y)$ with a symplectic embedding $(G_1,X_1)\rightarrow (\mathrm{GSp}(V_1,\psi_1),S^{\pm}_1)$. The natural morphism $(G_1,X_1)\rightarrow (G',X')$ induces, using similar notations as in \S\ref{conventions}, morphisms of Shimura variety $\Sh_{K_{1,p}}\rightarrow \Sh_{K'_{p}}$. We remark that $E'\subseteq E_1\supseteq E$.

By \S\ref{extend morp gp sch} again, the homomorphism $G_{1}\rightarrow G'$ extends to a homomorphism of group schemes $G_{1,\INTP}\rightarrow G'_{\INTP}$, and hence by Proposition \ref{extn morp on int}, the morphism of Shimura varieties $\Sh_{K_{1,p},E_1^p}\rightarrow \Sh_{K'_{p},E_1^p}$ extends to a morphism $\ES_{K_{1,p},O_1^p}\rightarrow \ES_{K'_{p},O_1^p}$. The component containing the image of $X^+\times\{\mathrm{id}_G\}$ is mapped to that containing $X'^+\times\{\mathrm{id}_{G'}\}$, and hence we have the induced morphism $\ES^+_{K_{1,p},O_1^p}\rightarrow \ES^+_{K'_{p},O_1^p}$. But by (1), we have a natural isomorphism $\ES^+_{K_{p},O_1^p}\cong \ES^+_{K_{1,p},O_1^p}$, and hence a morphism $\ES^+_{K_{p},O_1^p}\rightarrow \ES^+_{K'_{p},O_1^p}$. It is necessarily defined over $\mathtt{O}$, as both the integral models are defined over $\mathtt{O}$, and the morphism on generic fibers is defined over $\mathtt{E}$.
\end{proof}
%\begin{remark}
%\subsubsection{} \ \ \ \ \ \ \ \ \ $\xymatrix@C=1.3cm@R=0.5cm{
% &\widetilde{\ES}^+_{\Gamma_p}\ar[dl]_{\pi_K}\ar[dr]^{q_K}\ar[d]   \\
%\ES^+_{\Gamma_p}\ar@{}[ddr] |{\Box} \ar[d]_{f}  & \widetilde{\ES}^+_{\Gamma_p}\times^{G_{\INT_p}}G'_{\INT_p}\ar[l]\ar[d]^(0.45){\cong}& \mathrm{M}_{G,X}^{\mathrm{loc}}\ar[d]^{q_f}\\
%\ES'^+_{\Gamma'_p} & f^*\widetilde{\ES}'^+_{\Gamma'_p}\ar[ul]\ar[d]\ar[r] & \mathrm{M}_{G',X'}^{\mathrm{loc}}\\
%   &    \widetilde{\ES}'^+_{\Gamma'_p} \ar[ur]_{q'_{K'}}\ar[ul]^{\pi'_{K'}}       }$
%\end{remark}
\section[ab type]{Compatibility for parahoric models of abelian type}

Let $(G',X')$ be of Hodge type such that condition (\ref{conditon--Hodge type}) holds. Recall that as in (\ref{prob--Hodge to Hodge}), there are, a priori, two different ways to construct $\ES_{K'^\circ}$: one can construct $\ES_{K'^\circ}$ using \S\ref{para int setting} directly; or one could view it as a Shimura datum of abelian type, and follow the constructions in Theorem \ref{Thm--constuct int ab type}, by first choosing a Shimura datum $(G,X)$ as in Lemma \ref{lemma--choose of Hodge cover}, and then constructing the integral of $\Sh_{K'^\circ}$ via the integral model $\ES_{K^\circ}$ by Theorem \ref{Thm--constuct int ab type}.
\begin{theorem}\label{Thm--indep ab type}The above two constructions of $\ES_{K'^\circ}(G',X')$ coincide.
\end{theorem}
\begin{proof}
Let $(H,Y)\subseteq (G,X)$ and $(H',Y')\subseteq (G',X')$ be connected Shimura data of Hodge type. We first prove the statement under the additional assumption that $Y=Y'$ (as subsets of $X^{\mathrm{ad}}$). By the construction of $H$, $H\twoheadrightarrow H'$ is a central isogeny, and it extends to a homomorphism $H'_{\INT_p}\rightarrow H_{\INT_p}$. By Theorem \ref{result for conn shv} (2) (and also using the notations there), the morphism $\Sh^+_{\Gamma_p, \mathtt{E}}\rightarrow \Sh^+_{\Gamma'_p,\mathtt{E}}$ extends to a morphism $\ES^+_{\Gamma_p,\mathtt{O}}\rightarrow \ES^+_{\Gamma'_p,\mathtt{O}}$.

Notations as in the proof of Theorem \ref{result for conn shv} (2), we have a symplectic embedding $(G_1,X_1)\hookrightarrow (\mathrm{GSp}(V_1,\psi_1), S^\pm)$ whose homomorphism of groups extends to a closed immersion $G_{1,\INT_p}\hookrightarrow \mathrm{GL}(V_{1,\INT_p})$. Let $\widetilde{\ES}^+_{K_{p},\mathtt{O}}$ be the attached $G_{1,\INT_p}$-torsor on $\ES^+_{K_{p},\mathtt{O}}$. Then as in the proof of Theorem \ref{result for conn shv} (1), we have a commutative diagram
\subsubsection{} \ \ \ \ \ $\xymatrix@C=1.8cm@R=0.7cm{
\ES^+_{K_{p},\mathtt{O}}\ar@{}[ddr] |{\Box} \ar[d]_{f}  & \widetilde{\ES}^+_{K_{p},\mathtt{O}}\ar[d]\ar[l]_{\pi_K}\ar[r]^{q_K}& \mathrm{M}_{G,X}^{\mathrm{loc}}\ar[d]^{\cong}\\
\ES^+_{K'_{p},\mathtt{O}} & \widetilde{\ES}^+_{K_{p},\mathtt{O}}\times^{G_{1,\INT_p}}G'_{\INT_p}\ar[ul]\ar[d]\ar[r] & \mathrm{M}_{G',X'}^{\mathrm{loc}}\\
   &    \widetilde{\ES}^+_{K'_{p},\mathtt{O}} \ar[ur]_{q_{K'}}\ar[ul]^{\pi_{K'}}       }$\\
which implies that the morphism $\ES^+_{\Gamma_p,\mathtt{O}}\rightarrow \ES^+_{\Gamma'_p,\mathtt{O}}$ is \'{e}tale.

We claim that $\ES^+_{\Gamma'_p,\mathtt{O}}=\ES^+_{\Gamma_p,\mathtt{O}}/\Delta$, where $\Delta:=\mathrm{ker}(\mathscr{A}(G_{\INTP})^\circ\rightarrow \mathscr{A}(G'_{\INTP}))$. The morphism $\ES^+_{\Gamma_p,\mathtt{O}}\rightarrow \ES^+_{\Gamma'_p,\mathtt{O}}$ factors through $\ES^+_{\Gamma_p,\mathtt{O}}/\Delta$, as $\Delta$ acts trivially on $\Sh^+_{\Gamma'_p,\mathtt{E}}$ and hence trivially on $\ES^+_{\Gamma'_p,\mathtt{O}}$. The morphism $\ES^+_{\Gamma_p,\mathtt{O}}/\Delta\rightarrow \ES^+_{\Gamma'_p,\mathtt{O}}$ is an isomorphism on generic fibers, and induces isomorphisms of complete local rings at closed points in the special fibers. This implies that it is an open immersion. But by \cite[Theorem 4.6.23]{Paroh}, the induced morphism on special fibers is surjective, and hence is an isomorphism.

Let $\ES'_{K'^\circ_p}$ be the integral model of $\Sh_{K'^\circ_p}$ induced by $\ES_{K^\circ}$, i.e. $$\ES'_{K'^\circ_p}:=\big[[\ES^+_{\Gamma_p,\mathtt{O}}\times \mathscr{A}(G'_{\INTP})]/\mathscr{A}(G_{\INTP})^\circ\big]^{|\mathrm{J}|}.$$ We will compare $\ES'_{K'^\circ_p}$ and $\ES_{K'^\circ_p}$. By Proposition \ref{extn morp on int}, the $G^{\mathrm{ad}\circ}(\INTP)^+$-action on $\Sh_{K'^\circ_p}$ extends to $\ES_{K'^\circ_p}$. But then the $\mathscr{A}(G'_{\INTP})$-action, and hence the $\mathscr{A}(G_{\INTP})^\circ$-action, on $\Sh_{K'^\circ_p,\mathtt{E}}$ extends to $\ES_{K'^\circ_p, \mathtt{O}}$. The $\mathscr{A}(G_{\INTP})^\circ$-action leaves $\ES_{K'^\circ_p, \mathtt{O}}$ stable, so we have a morphism $\big[[\ES^+_{\Gamma'_p,\mathtt{O}}\times \mathscr{A}(G'_{\INTP})]/(\mathscr{A}(G_{\INTP})^\circ/\Delta)\big]^{|\mathrm{J}|}\rightarrow \ES_{K'^\circ_p}$ which is necessarily an isomorphism. The morphism $\ES^+_{\Gamma_p,\mathtt{O}}\rightarrow \ES^+_{\Gamma'_p,\mathtt{O}}$ then induces a morphism  $[\ES^+_{\Gamma_p,\mathtt{O}}\times\mathscr{A}(G'_{\INTP})]^{|\mathrm{J}|} \longrightarrow [\ES^+_{\Gamma'_p,\mathtt{O}}\times \mathscr{A}(G'_{\INTP})]^{|\mathrm{J}|}$, and hence an isomorphism $\ES'_{K'^\circ_p, \mathtt{O}}\rightarrow \ES_{K'^\circ_p,\mathtt{O}}$ which descends to $O_{E'}$.

In general, by the real approximation, there is a $g\in G^{\mathrm{ad}}(\RAT)$ such that ${}^gY\subseteq Y'$. Here we write ${}^g(-)$ for the action of induced by conjugation. It also induces an isomorphism of Shimura data $(G,X)\rightarrow (G,{}^gX)$. Let ${}^gG_{\INT_p}$ be the affine scheme attached to the image of $\xymatrix{O_{G_{\INT_p}}\subseteq O_{G_{\RAT_p}}\ar[r]^(0.6){{}^{g^{-1}}(-)}& O_{G_{\RAT_p}}},$ it is a smooth group scheme with $({}^gG_{\INT_p})(\INT_p)={}^gK_p$. By Proposition \ref{extn morp on int}, the isomorphism of Shimura varieties $\Sh_{K}(G,X)\stackrel{\cong}{\longrightarrow} \Sh_{{}^gK}(G,{}^gX)$ extends to an isomorphism of integral models $\ES_{K}(G,X)\stackrel{\cong}{\longrightarrow} \ES_{{}^gK}(G,{}^gX)$. The same statement holds if we pass to parahoric level. In particular, it suffices to prove the theorem when $Y=Y'\subseteq X'$.
\end{proof}
Now let $(G_4,X_4)$ be of abelian type such that condition (\ref{assump--ab type}) holds. As in (\ref{prob--Ab type}), the construction of $\ES_{K^\circ_4}(G_4,X_4)$ in Theorem \ref{Thm--constuct int ab type} requires to choose a Shimura datum $(G,X)$ as in Lemma \ref{lemma--choose of Hodge cover} first. It is natural to ask whether $\ES_{K^\circ_4}(G_4,X_4)$ depends on the choices of $(G,X)$s. If $(G_4,X_4)$ of \emph{Hodge type} such that condition (\ref{conditon--Hodge type}) holds, the previous theorem already claims that it is independent of choices of the $(G,X)$s. Using similar proof, one can show the followings.
\begin{theorem}\label{Thm--indep ab type 2}The integral model $\ES_{K^\circ_4}(G_4,X_4)$ is independent of choices of the $(G,X)$s as above.
\end{theorem}
\begin{proof}
Let $(G,X)$ and $(G',X')$ be two choices of Hodge type Shimura data as above. By our choices (see Lemma \ref{lemma--choose of Hodge cover}), $G^\mathrm{der}=G'^{\mathrm{der}}$, and $\Sh_{K^\circ_p}$ (resp. $\Sh_{K'^\circ_p}$) is also defined over $E_4$, the $p$-adic reflex field of $(G_4, X_4)$. Without loss of generality, we can assume that they are extensions of the same connected Shimura data $(H,Y)$ with $H=G^\mathrm{der}$ and $Y\subseteq X_4$. By Theorem \ref{result for conn shv} (1), the natural identification $\Sh^+_{K^\circ_p}\rightarrow \Sh^+_{K'^\circ_p,}$ extends to an isomorphism $\ES^+_{K^\circ_p}\rightarrow \ES^+_{K'^\circ_p}$.

Noting that $\mathscr{A}(G_{\INTP})^\circ=\mathscr{A}(G'_{\INTP})^\circ$, the quotients
\begin{align*}
\ES_{K^\circ_{4,p}}&:=\big[[\ES^+_{K^\circ_p}\times \mathscr{A}(G_{4,\INTP})]/\mathscr{A}(G_{\INTP})^\circ\big]^{|\mathrm{J}|}\\
\text{and\ \ \ \ }\ES'_{K^\circ_{4,p}}&:=\big[[\ES^+_{K'^\circ_p}\times \mathscr{A}(G_{4,\INTP})]/\mathscr{A}(G'_{\INTP})^\circ\big]^{|\mathrm{J}|}
\end{align*}
are then integral models induced by $\ES_{K^\circ_p}$ and $\ES_{K'^\circ_p}$ respectively. The identification $[\ES^+_{K^\circ_p}\times\mathscr{A}(G_{4,\INTP})]^{|\mathrm{J}|} \longrightarrow [\ES^+_{K'^\circ_p}\times \mathscr{A}(G_{4,\INTP})]^{|\mathrm{J}|}$ then induces an isomorphism $\ES_{K^\circ_{4,p}}\rightarrow \ES'_{K^\circ_{4,p}}$.
\end{proof}

\subsection{} The formulism of parahoric integral models is functorial in many cases. Let $f:(G_4,X_4)\rightarrow (G'_4,X'_4)$ be a morphism of Shimura data with $G_4$ and $G'_4$ satisfying condition (\ref{assump--ab type}). Considering the natural morphisms of Bruhat-Tits buildings $\mathcal {B}(H,\mathbb{Q}_p)\rightarrow \mathcal {B}(G_4^{\mathrm{der}},\mathbb{Q}_p)\subseteq \mathcal {B}(G_4,\mathbb{Q}_p)\rightarrow \mathcal {B}(G^{\mathrm{ad}},\mathbb{Q}_p)$, we have natural homomorphisms $H_{\INT_p}\rightarrow  G_{4,\INT_p}\rightarrow  G^\mathrm{ad}_{\INT_p}$. Similarly, we have homomorphisms $H'_{\INT_p}\rightarrow  G'_{4,\INT_p}\rightarrow  G'^{\mathrm{ad}}_{\INT_p}$. We will need the following condition on $f$.

\subsubsection{}\label{condi--morp ab type 1} The homomorphism $f:G_4\rightarrow G'_4$ induces a commutative diagram
$$\xymatrix{
H\ar[r]\ar[d]& G_4\ar[r]\ar[d]^f& G^\mathrm{ad}\ar[d]\\
H'\ar[r]& G'_4\ar[r]& G'^{\mathrm{ad}},}$$
and the homomorphisms $H\rightarrow H'$, $G_4\rightarrow G'_4$ and $G^\mathrm{ad}\rightarrow G'^{\mathrm{ad}}$ extend to homomorphisms $H_{\INT_p}\rightarrow H'_{\INT_p}$, $G^\circ_{4,\INT_p}\rightarrow G'^\circ_{4,\INT_p}$ and $G^{\mathrm{ad}\circ}_{\INT_p}\rightarrow G'^{\mathrm{ad}\circ}_{\INT_p}$ respectively.

%$$\xymatrix{
%H_{\INT_p}\ar[r]\ar[d]& G_{4,\INT_p}\ar[r]\ar[d]^{f_{\INT_p}}& G^\mathrm{ad}_{\INT_p}\ar[d]\\
%H'_{\INT_p}\ar[r]& G'_{4,\INT_p}\ar[r]& G'^{\mathrm{ad}}_{\INT_p}.}$$

\begin{proposition}\label{prop--funct ab type}
Notations as above. If $f$ satisfies condition (\ref{condi--morp ab type 1}), then the morphism $\Sh_{K^\circ_4}\rightarrow \Sh_{K'^\circ_4}$ extends to a morphism $\ES_{K^\circ_4}\rightarrow \ES_{K'^\circ_4}$.
\end{proposition}
\begin{proof}
Choosing $(G,X)$ and $(G',X')$ as in Lemma \ref{lemma--choose of Hodge cover} for $(G_4,X_4)$ and $(G'_4,X'_4)$ respectively, we get $\ES^+_{K^\circ_p}$ and $\ES^+_{K'^\circ_p}$ if we also choose $X^+\subseteq X$ and $X'^+\subseteq X'$. Without loss of generality, we can assume that $X^+\subseteq X_4$, $X'^+\subseteq X'_4$ and $f(X^+)\subseteq X'^+$. By Theorem \ref{result for conn shv} (2), the morphism $\Sh^+_{K^\circ_p}\rightarrow \Sh^+_{K'^\circ_p}$ extends to a morphism $\ES^+_{K^\circ_p}\rightarrow\ES^+_{K'^\circ_p}$.

By our assumptions (and the discussions after Lemma \ref{lemma--choose of Hodge cover}), we have natural homomorphisms $\mathscr{A}(G_{\INTP})^\circ\rightarrow \mathscr{A}(G'_{\INTP})^\circ$,  $\mathscr{A}(G_{4,\INTP})\rightarrow \mathscr{A}(G'_{4,\INTP})$ and a natural map  $\mathscr{A}(G)^\circ\backslash\mathscr{A}(G_4)/K_{4,p}^\circ\rightarrow \mathscr{A}(G')^\circ\backslash\mathscr{A}(G'_4)/K'^\circ_{4,p}.$ They induce a map $\mathrm{J}\rightarrow \mathrm{J}'$, and hence a morphism $$\big[[\Sh_{K^\circ_p}^+\times\mathscr{A}(G_{4,\INTP})]/\mathscr{A}(G_{\INTP})^\circ\big]^{|\mathrm{J}|} \longrightarrow \big[[\Sh^+_{K'^\circ_p}\times \mathscr{A}(G'_{4,\INTP})]/\mathscr{A}(G'_{\INTP})^\circ\big]^{|\mathrm{J}'|},$$
which, by the proof of \cite[Lemma 4.6.13]{Paroh}, can be identified with the morphism $\Sh_{K^\circ_{4,p}}\rightarrow \Sh_{K'^\circ_{4,p}}$. But then the morphism $$\big[[\ES_{K^\circ_p}^+\times\mathscr{A}(G_{4,\INTP})]/\mathscr{A}(G_{\INTP})^\circ\big]^{|\mathrm{J}|} \longrightarrow \big[[\ES^+_{K'^\circ_p}\times \mathscr{A}(G'_{4,\INTP})]/\mathscr{A}(G'_{\INTP})^\circ\big]^{|\mathrm{J}'|}$$
induced by $\ES_{K^\circ_p}^+\rightarrow\ES^+_{K'^\circ_p}$ gives a morphism $\ES_{K^\circ_{4,p}}\rightarrow \ES_{K'^\circ_{4,p}}$ extending the one on the generic fibers.
\end{proof}
\subsection{}\label{remark--when funct works}  There are situations when condition (\ref{condi--morp ab type 1}) holds. For simplicity, \emph{we assume that} $(G_4^{\mathrm{ad}},X_4^{\mathrm{ad}})$ \emph{has no factors of type} $\mathrm{D}^{\mathbb{H}}$. In particular, $H$ is the universal covering of $G_4^{\mathrm{der}}$.
\subsubsection{}\label{case--change of parah} (Morphisms induced by changes in parahoric) We assume in this case that $(G_4,X_4)=(G'_4,X'_4)$. Let $K^\circ_{4,p}$ and $K'^\circ_{4,p}$ be parahoric subgroups attached to points $\mathfrak{b}$ and $\mathfrak{b}'$ in $\mathcal {B}(G_4,\mathbb{Q}_p)$ respectively. If the open facet containing $\mathfrak{b}'$ is in boundary of that containing $\mathfrak{b}$, then we have natural homomorphisms $G^\circ_{4,\INT_p}\rightarrow G'^\circ_{4,\INT_p}$ and $G^{\mathrm{ad}\circ}_{\INT_p}\rightarrow G'^{\mathrm{ad}\circ}_{\INT_p}$.

Replacing $\mathfrak{b}$ (resp. $\mathfrak{b}'$) by another point in its open facet if necessary, we have then $H^\circ_{\INT_p}= H_{\INT_p}$ (resp. $H'^\circ_{\INT_p}= H'_{\INT_p}$), as $H$ is simply connected. In particular, we have a natural homomorphism $H_{\INT_p}\rightarrow  H'_{\INT_p}$. To sum up, condition (\ref{condi--morp ab type 1}) holds.

\subsubsection{}\label{case--induced by surj} (Morphisms induced by surjections) We assume in this case that $f:G_4\rightarrow G'_4$ is surjective. It induces commutative diagrams
\begin{center}
$\xymatrix{
H\ar[r]\ar[d]& G_4\ar@{->>}[r]\ar@{->>}[d]& G^\mathrm{ad}_4\ar@{->>}[d]\\
H'\ar[r]& G'_4\ar@{->>}[r]& G'^{\mathrm{ad}}_4}$ \ \ \ \ and \ \ \ \ $\xymatrix{
\mathcal {B}(G_4,\mathbb{Q}_p)\ar[r]\ar[d]& \mathcal {B}(G^\mathrm{ad}_4,\mathbb{Q}_p)\ar[d]\\
\mathcal {B}(G'_4,\mathbb{Q}_p)\ar[r]& \mathcal {B}(G'^{\mathrm{ad}}_4,\mathbb{Q}_p),}$
\end{center}
where the second diagram is induced by the first one by functoriality (see e.g. \cite[Theorem 2.1.8]{funct BT}).

%In particular, $(G'^{\mathrm{ad}}_4,X'^{\mathrm{ad}}_4)$ has no factors of type $\mathrm{D}^{\mathbb{H}}$, and hence $H'$ is also the universal covering of $G'^{\mathrm{der}}_4$.

Let $K^\circ_{4,p}$ be attached to $\mathfrak{b}\in \mathcal {B}(G_4,\mathbb{Q}_p)$. If $K'^\circ_{4,p}$ is attached to the image of $\mathfrak{b}$ in $\mathcal {B}(G'_4,\mathbb{Q}_p)$, the homomorphisms $H\rightarrow H'$, $G_4\rightarrow G'_4$ and $G^\mathrm{ad}\rightarrow G'^{\mathrm{ad}}$ extend to homomorphisms $H_{\INT_p}\rightarrow H'_{\INT_p}$, $G_{4,\INT_p}\rightarrow G'_{4,\INT_p}$ and $G^{\mathrm{ad}}_{\INT_p}\rightarrow G'^{\mathrm{ad}}_{\INT_p}$ respectively. In particular, condition (\ref{condi--morp ab type 1}) holds.

\

\end{document}